\documentclass[10pt,a4paper]{article}
\usepackage{fullpage}
\usepackage{amsfonts,amsmath,amssymb}
\usepackage{graphicx}\usepackage{sectsty}
\newtheorem{theorem}{Theorem}[section]
\newtheorem{lemma}[theorem]{Lemma}
\newtheorem{definition}[theorem]{Definition}
\newtheorem{proposition}[theorem]{Proposition}
\newtheorem{corollary}[theorem]{Corollary}
\newenvironment{proof}[1][Proof]{\begin{trivlist}
\item[\hskip \labelsep {\bfseries #1}]}{\end{trivlist}}

\newtheorem{remark}[theorem]{Remark}
\numberwithin{equation}{section}

\sectionfont{\large}

\begin{document}
\title{An Uniqueness Result on Spherically Stratified Media with Interior Transmission Eigenvalues }
\author{Lung-Hui Chen$^1$}\maketitle\footnotetext[1]{Department of
Mathematics, National Chung Cheng University, 168 University Rd.
Min-Hsiung, Chia-Yi County 621, Taiwan. Email:
mr.lunghuichen@gmail.com/lhchen@math.ccu.edu.tw. Fax:
886-5-2720497.}
%------------------------------------------------------------------------%
\begin{abstract}
Given a set of transmission eigenvalues, we describe the
connection between such a set and the indicator functions in
entire function theory. The indicator functions control the
asymptotic growth rate of the solution of the Sturm-Liouville
problem which has an uniqueness in the inverse spectral theory.
Accordingly, the set of transmission eigenvalues has an inverse
spectral property.
\end{abstract}
\section{Introduction and the Main Result}
In this paper, we consider the stationary scattering problem
\begin{eqnarray}\label{1.1}
\left\{%
\begin{array}{ll}
\Delta u+k^2n(x)u=0, & \mbox{ in }\mathbb{R}^3; \\
    u=u^s+u^i;\\
    \lim_{r\rightarrow\infty}r\{\frac{\partial u^s}{\partial r}-iku^s\}=0, \\
\end{array}%
\right.
\end{eqnarray}
where $u^i$ is an entire solution of the Helmholtz equation, $k$
is the wave number and
$n\in\mathcal{C}^1(0,a)\cap\mathcal{H}^2(0,a)$ is specified
spherically stratified in $B$, where
$B:=\{\mathbb{R}^3|\,|x|<a\}$, such that $n(x)=n(r)>0$ and
$n(a)=1$. We may ask that is there any incident fields $u^i$ such
that the scattered field $u^s$ is identically zero? The answer is
positive provided that there exists a nontrivial solution to the
following interior transmission problem: $k\in\mathbb{C}$,
$w,v\in\mathcal{L}^2(B)$, $w-v\in\mathcal{H}^2_0(B)$,
\begin{eqnarray}\label{12}
\left\{%
\begin{array}{ll}
    \Delta w+k^2n(r)w=0,  & \hbox{ in }B; \\
    \Delta v+k^2v=0 & \hbox{ in }B; \\
    w=v, & \hbox{ on }\partial B; \\
    \frac{\partial w}{\partial r}=\frac{\partial v}{\partial r}& \hbox{ on }\partial B.\\
\end{array}%
\right.
\end{eqnarray}
For spherical perturbations as the one given by~(\ref{1.1}), we
consider the spherical harmonics which we refer to Colton and
Kress \cite{Colton} for a theory in inverse problems. In
particular, we let $Y_l$ be the solution of following equation.
\begin{equation}
Y_l''+\frac{2}{r}Y_l'+\{k^2n(r)-\frac{l(l+1)}{r^2}\}Y_l=0,
\end{equation}
such that
\begin{equation}
\lim_{r\rightarrow0}\{Y_l(r)-j_l(kr)\}=0,
\end{equation}
where $j_l$ is a spherical Bessel function. The existence of the
nontrivial solution of the interior transmission
problem~(\ref{12}) is implied by the existence of the nontrivial
constants $a_l$, $b_l$ for some $k$, $l$ to the following system
of equations.
\begin{eqnarray}
\left\{%
\begin{array}{ll}
    a_lY_l(a)-b_lj_l(ka)=0;\\
    a_lY_l'(a)-b_lkj_l'(ka)=0.\\
\end{array}%
\right.
\end{eqnarray}
Equivalently, we are looking for the zeros of the following
functional determinant.
\begin{equation}
d_l(k):=\det\left(%
\begin{array}{cc}
  Y_l(a) & -j_l(ka) \\
  Y_l'(a) & -kj_l'(ka) \\
\end{array}%
\right).
\end{equation}
These zeros of the functional determinant $d_l(k)$ are the
interior transmission eigenvalues of~(\ref{12}). In particular,
when $l=0$, we consider the zeros of
\begin{equation}\label{1.7}
d_0(k):=\det\left(%
\begin{array}{cc}
  Y_0(a) & -j_0(ka) \\
  Y_0'(a) & -kj_0'(ka) \\
\end{array}%
\right),
\end{equation}
where $Y_0=\frac{y(r)}{r}$ and $y(r)$ satisfies
\begin{equation}\label{1.8}
y''+k^2n(r)y=0,\,y(0)=0,\,y'(0)=1,
\end{equation}
which is well-known that there exists an unique solution
to~(\ref{1.8}) to every $k\in\mathbb{C}$. Hence, we may see
$y=y(x;k)$. $j_0(kr)=\frac{\sin{kr}}{kr}$. Moreover, we have the
asymptotics of $y(r;k)$ and $y'(r;k)$:
\begin{eqnarray}\label{1.9}
y(x;k)=\frac{1}{[n(0)n(r)]^{\frac{1}{4}}k}\sin(k\int_0^r[n(r)]^{\frac{1}{2}}dr)[1+O(\frac{1}{|k|})],\,\forall
k\in \mathbb{C}\setminus\{0i+\mathbb{R}\}.
\end{eqnarray}
Similarly,
\begin{eqnarray}\label{1.10}
y'(x;k)=[n(r)/n(0)]^{\frac{1}{4}}\cos({k}\int_0^r[n(r)]^{\frac{1}{2}}dr)[1+O(\frac{1}{|k|})],\,\forall
k\in \mathbb{C}\setminus\{0i+\mathbb{R}\}.
\end{eqnarray}
We refer such asymptotics to \cite[Proposition 2.3]{Aktosun}. Such
asymptotic expansions are classic in spectral theory. See
P\"{o}schel and Trubowitz \cite{Poschel} and Naimark
\cite{Naimark}.
\begin{definition}We define
$s:=a-\int_0^a[n(r)]^{\frac{1}{2}}dr$ and
$b:=\int_0^a[n(r)]^{\frac{1}{2}}dr$.
\end{definition}
In this paper, we assume either
\begin{equation}\label{H1}
a>b\mbox{ or }a<b,
\end{equation}
simultaneously for each $n(r)$'s. The following asymptotic
behavior holds.
\begin{equation}\label{1.12}
d_0(k)=\frac{1}{a^2k[n(0)]^{1/4}}\sin{k(a-b)}+O(\frac{1}{k^2}),\forall
k\in 0i+\mathbb{R}.
\end{equation}
Such asymptotic expansion is classical in spectral theory. See
Colton and Kress \cite{Colton}. Let us consider its behavior in
$\mathbb{C}$. From~(\ref{1.9}) and~(\ref{1.10}),
\begin{eqnarray}\label{1.13}\nonumber
\frac{y(a)}{a}[-kj_0'(ka)]&=&[\frac{B\sin ka\sin
kb}{a^3k^2}-\frac{B\cos ka\sin kb}{a^2k}][1+O(\frac{1}{|k|})],\,
B:=\frac{1}{[n(0)n(a)]^{\frac{1}{4}}};\\
j_0(ka)(\frac{y(r)}{r})'|_{r=a}&=&[\frac{C\sin{ka}\cos{kb}}{a^2k}-\frac{B\sin{ka}\sin{kb}}{a^3k^2}][1+O(\frac{1}{|k|})],\,
C:=[\frac{n(a)}{n(0)}]^{\frac{1}{4}}.
\end{eqnarray}
Hence, we have the asymptotics for $d_0(k)$.
\begin{eqnarray}
d_0(k)&=&j_0(ka)(\frac{y(r)}{r})'|_{r=a}+\frac{y(a)}{a}[-kj_0'(ka)]\\\label{1.15}
&=&\frac{1}{a^2k[n(0)]^{\frac{1}{4}}}[\sin
k(a-b)][1+O(\frac{1}{|k|})],\,\forall k\in
\mathbb{C}\setminus\{0i+\mathbb{R}\}.
\end{eqnarray}
\par
In Aktosun, Gintides and Papanicolaou \cite[p.5]{Aktosun}, they
have shown if $\int_0^a[n(r)]^{\frac{1}{2}}dr<a$, then the
transmission eigenvalues corresponding to spherically symmetric
solutions of the interior transmission problem uniquely determine
the $n(r)$. Furthermore, in \cite{Aktosun}, it is shown if
$d_0(k)\equiv0$ for $\lambda\in\mathbb{C}$, then $n(r)\equiv1$ in
$[0,a]$. It is also discussed that the signs of the quantity $a-b$
plays a role in the inverse spectral theory in \cite{Aktosun}.
Furthermore, in Cakoni, Colton and Gintides \cite[Theorem
2.1]{Cakoni}, they have shown if $n(0)$ is given and $n(r)>1$,
then $n(r)$ is uniquely determined from some knowledge of the
transmission eigenvalues. It is expected among mathematicians,
say, \cite{Colton3}, that such an uniqueness holds for the
condition $n(r)>1$. In this paper, we show that only the interior
transmission eigenvalues near the real axis are needed to
determine the $n(r)$. In addition, we propose another qualitative
description on the counting function to the zeros of $d_l(z)$. In
the spectral theory of Sturm-Liouville, such a qualitative
description on the growth rate of zeros of $d_l$ is connected to
the inverse problem on finding index $n(r)$. For such a
connection, we refer to McLaughlin and Polyakov \cite{Mc} which is
based on P\"{o}schel and Trubowitz's work in \cite{Poschel}. In
\cite{Mc}, there is an argument on the qualitative description for
the zeros of $d_0(z)$.
\par
Firstly, we need some vocabulary from entire function theory. We
refer to the Levin's book \cite{Levin,Levin2}.
\begin{definition}
Let $f(z)$ be an entire function. Let
$M_f(r):=\max_{|z|=r}|f(z)|$. An entire function of $f(z)$ is said
to be a function of finite order if there exists a positive
constant $k$ such that the inequality
\begin{equation}
M_f(r)<e^{r^k}
\end{equation}
is valid for all sufficiently large values of $r$. The greatest
lower bound of such numbers $k$ is called the order of the entire
function $f(z)$. By the type $\sigma$ of an entire function $f(z)$
of order $\rho$, we mean the greatest lower bound of positive
number $A$ for which asymptotically we have
\begin{equation}
M_f(r)<e^{Ar^\rho}.
\end{equation}
That is
\begin{equation}
\sigma:=\limsup_{r\rightarrow\infty}\frac{\ln
M_f(r)}{r^\rho}.
\end{equation}  If $0<\sigma<\infty$, then we say
$f(z)$ is of normal type or mean type.
\end{definition}
\begin{definition}
If an entire function $f(z)$ is of order one and of normal type,
then we say it is an entire function of exponential type $\sigma$.
\end{definition}
\begin{definition}
Let $\rho\in\mathbb{R}$ and
$\rho(r):\mathbb{R}^+\rightarrow\mathbb{R}^+$. We say $\rho(r)$ is
a proximate order to $\rho$ if
\begin{equation}
\lim_{r\rightarrow\infty}\rho(r)=\rho\geq0;\,\lim_{r\rightarrow\infty}r\rho'(r)\ln
r=0.
\end{equation}
\end{definition}
\begin{definition}\label{1.5}
Let $f(z)$ be an integral function of finite order in the angle
$[\theta_1,\theta_2]$, we call the following quantity as the
generalized indicator of the function $f(z)$.
\begin{equation}
h_f(\theta):=\limsup_{r\rightarrow\infty}\frac{|f(re^{i\theta})|}{r^{\rho(r)}},\,\theta_1\leq\theta\leq\theta_2,
\end{equation}
where $\rho(r)$ is some proximate order.
\end{definition}
We review two inequalities for indicator functions.
\begin{eqnarray}\label{1.21}
&&h_{fg}(\theta)\leq h_{f}(\theta)+h_g(\theta);\\
&&h_{f+g}\leq\max\{h_f(\theta),h_g(\theta)\}. \label{1.22}
\end{eqnarray}
We can find these in \cite[p.51]{Levin}.

The order and the type of an integral function in an angle can be
defined similarly. The connection between the indicator
$h_f(\theta)$ and its type $\sigma$ can be specified by the
following theorem.
\begin{lemma}[Levin \cite{Levin},\,p.72]\label{1.7}
The maximum value of the indicator $h_f(\theta)$ of the function
$f(z)$ on the interval $\alpha\leq\theta\leq\beta$ is equal to the
type $\sigma_f$ of this function inside the angle $\alpha\leq\arg
z\leq\beta$.
\end{lemma}
\begin{lemma}
Let $f(z)=C\sin{Az}$, where $A$, $C$ are constants. $f(z)$ is an
entire function of exponential type $|A|$.
\end{lemma}
\begin{proof}
It suffices to see $\frac{\sin{z}}{z}=\frac{e^{iz}-e^{-iz}}{2iz}$
and
\begin{equation}
h_{\frac{\sin
z}{z}}(\theta)=\limsup_{r\rightarrow\infty}\frac{\ln|\frac{\sqrt{e^{2r\sin\theta}
+e^{-2rsin\theta}-2\cos{2(r\cos\theta)}}}{2r}|}{r}=|sin\theta|.
\end{equation}$\Box$
\end{proof}
For our indicator for $d_0(z)$.
\begin{lemma}\label{1.8}
Let $b=\int_0^a[n(r)]^{\frac{1}{2}}dr$. Then, given $d_0(z)$ as
in~(\ref{1.12}), it is an entire function of exponential type
$|a-b|$ with indicator function
\begin{equation}
h_{d_0}(\theta)=|\sin{\theta}||a-b|.\end{equation}
\end{lemma}
\begin{proof}
Let $\eta:=za-zb$. We compute that
\begin{eqnarray}\nonumber
|\sin
\eta|^2&=&\frac{1}{4}[e^{\eta+\overline{\eta}}-e^{\eta-\overline{\eta}}-e^{-\eta+\overline{\eta}}+e^{-\eta-\overline{\eta}}]\\
&=&\frac{1}{4}[e^{-2ya+2yb}-e^{2ixa-2ixb}-e^{-2ixa+2ixb}+e^{2ya-2yb}]\nonumber\\
&=&\frac{1}{4}[e^{-2ya+2yb}+e^{2ya-2yb}-2\cos(2xa-2xb)].
\end{eqnarray}
Following from~(\ref{1.15}) that
\begin{equation}
h_{d_0}(\theta)=\limsup_{r\rightarrow\infty}\frac{\ln|\sin\eta|}{r}=\limsup_{r\rightarrow\infty}\frac{|y||a-b|}{r}
=|\sin\theta||a-b|,\,\theta\neq0,\pi.
\end{equation}
We see that $h_{d_0}(\theta)$ is continuous in $[0,2\pi]$, Levin
\cite[p.54]{Levin}, we conclude
\begin{equation}
 h_{d_0}(\theta)=|a-b||\sin\theta|,\,\theta\in[0,2\pi].
\end{equation}
Furthermore, $d_0(z)$ is an entire function of exponential growth
due to the definition~(\ref{1.7}) and the fact that spherical
Bessel function $j_0(z)=\frac{\sin z}{z}$ and $y(r)$ are entire
functions of exponential type.

We see that $h_{d_0}(\theta)$ attains its nontrivial maximum
$|a-b|$ at $\theta=\pm\frac{\pi}{2}$. $\Box$
\end{proof}

\par
The following is the main theorem of this paper.
\begin{theorem}\label{19}
Let the $0$-th functional determinant $d_0(z)$ be defined as
above. Then, the zeros of $d_0(z)$ inside the angular wedge
\begin{equation}
\Sigma_\epsilon:=\{k\in\mathbb{C}|\,-\epsilon\leq\arg{k}\leq\epsilon,\,\pi-\epsilon\leq\arg{k}\leq\pi+\epsilon\},\,\forall\epsilon>0,
\end{equation}
uniquely determines between the indices $n(r)$'s provided they are
of the same value at $r=0$.
\end{theorem}

\section{M.L. Cartwright's Theorem}
The foundation of this paper is the following theorem in
Cartwright \cite[p.538]{Cartwright2}$^2$.
\begin{theorem}\label{2.8} Suppose that $f(z)$ is an integral
function of order $\rho=1$ with the following Hadamard's
representation:
\begin{equation}\label{2.1}
f(z)=z^m\exp\{c_0+c_1z\}\prod_{n=1}^\infty(1-\frac{z}{z_n})\exp\{\frac{z}{z_n}\}
\end{equation}
and that the indicator
\begin{equation}\label{2.2}
h(\theta)=\max\{A\cos\theta,B\cos\theta\},
\end{equation}
where $B\geq A$, $B>0$. Then, the following asymptotics hold.
\begin{equation}
\lim_{r\rightarrow\infty}\frac{n(r,-\frac{\pi}{2}+\delta,\frac{\pi}{2}-\delta)}{r^{\rho(r)}}=0;
\end{equation}
\begin{equation}
\lim_{r\rightarrow\infty}\frac{n(r,\frac{\pi}{2}+\delta,\frac{3\pi}{2}-\delta)}{r^{\rho(r)}}=0;
\end{equation}
\begin{equation}\label{2.17}
n(r,\pm\frac{\pi}{2}-\delta,\pm\frac{\pi}{2}+\delta)\sim\frac{1}{2\pi}\{B(r)-A(r)\}r^{\rho(r)},
\end{equation}
\begin{equation}\label{2.18}
c_1+\sum_{|z_n|\leq r}\frac{1}{z_n}\sim
\frac{1}{2}\{A(r)+B(r)\}r^{\rho(r)-1},
\end{equation}
where $A(r)$ and $B(r)$ are real functions such that $B(r)\geq
A(r)$,
\begin{equation}
\liminf_{r\rightarrow\infty}A(r)=A;\,\limsup_{r\rightarrow\infty}B(r)=B;
\end{equation}
\begin{equation}
\lim_{r\rightarrow\infty}\{A(r)-A(\eta
r)\}=0;\,\lim_{r\rightarrow\infty}\{B(r)-B(\eta
r)\}=0,\,\forall\eta>0,
\end{equation}
where $\rho(r)$ is a proximate order to $\rho$.
\end{theorem}
\begin{remark}We may choose $A=-B$ for this note. We also note
that $d_0(z)$ is bounded on $0i+\mathbb{R}$. Cartwright's paper
may be hard to obtained. We consider the Levin's book on functions
of class C in \cite{Levin2} for some modern reference and backup
theory source. To answer some mathematicians' question, rotating
the functional $d(z)$ by $\frac{\pi}{2}$ won't alter the nature of
the distribution of the zeros. Such a picture is clear in the
discussion in Levin \cite[p.127]{Levin2}. Our functional $d_l(z)$
is trivially of class C as already discussed in Chen \cite{Chen}.
Two conditions listed in \cite[p.115]{Levin2} can be justified by
the fact that $d_l(z)$ is bounded over the real axis. We compare
with the indicator function appears in~(\ref{2.2}), the one in
Lemma \ref{1.8} and the one in Levin \cite[p.126]{Levin2}. We see
all these three functions are consistent with each other. Another
examination on~(\ref{3.4}),~(\ref{37}),~(\ref{2.18})
and~(\ref{316}) yields the same indicator functions.
\end{remark}
\begin{corollary}
Let $n(d^j,r,\alpha,\beta)$ be the number of zeros of $d^j(z)$
 that are located in the closed cone
$\{z\in\mathbb{C}|\alpha\leq\arg z\leq\beta,\,|z|\leq r\}$ and
\begin{equation}\label{2.9}
\Delta^j(\alpha,\beta):=\lim_{r\rightarrow\infty}\frac{n(d^j,r,\alpha,\beta)}{r};\,\Delta^j(\beta)-\Delta(\alpha):=\Delta(\alpha,\beta)+C,
\end{equation}
where $j=1,2$, and $C$ is some constant. Then,
\begin{eqnarray}
&&\Delta^j(\delta,\pi-\delta)=0;\label{2.10}\\
&&\Delta^j(\pi+\delta,-\delta)=0;\label{2.11}\\
&&\Delta^j(-\delta,\delta)>0;\label{2.12}\\
&&\Delta^j(\pi-\delta,\pi+\delta)>0;\label{2.13}\\
&&\Delta^1(-\delta,\delta)=\Delta^2(-\delta,\delta)\neq0;\label{2.14}\\
&&\Delta^1(\pi-\delta,\pi+\delta)=\Delta^2(\pi-\delta,\pi+\delta)\neq0.\label{2.15}
\end{eqnarray}
Accordingly, let $E$ be the set of points of discontinuity of the
function $\Delta^j(\psi)$. Then, $E=\{0,\pi\}$.
\end{corollary}
\begin{proof}
This is only Cartwright's theory. The indicator function
$h_{d^j}(\theta)$ is computed in Lemma \ref{1.8}. Combining
with~(\ref{2.2}), we prove the corollary. $\Box$
\end{proof}

\section{ Proof of Theorem 1.9}
Let $d(z):=d_0(z)$. From~(\ref{2.1}), suppose that
\begin{equation}\label{3.1}
d(z)=z^m\exp\{c_0+c_1z\}\prod_{n=1}^\infty(1-\frac{z}{z_n})\exp\{\frac{z}{z_n}\}.
\end{equation}
Letting $z:=re^{i\theta}$.
\begin{equation}
\delta(r):=c_1+\sum_{|z_n|\leq r}\frac{1}{z_n}.
\end{equation}
\begin{equation}\label{3.3}
f_r(z):=\prod_{|z_n|\leq
r}(1-\frac{z}{z_n})\prod_{|z_n|>r}(1-\frac{z}{z_n})e^{\frac{z}{z_n}}.
\end{equation}
For each index $n^j(r)$, we use
\begin{equation}
f_r^j(z):=\prod_{|z_n^j|\leq
r}(1-\frac{z}{z_n^j})\prod_{|z_n^j|>r}(1-\frac{z}{z_n^j})e^{\frac{z}{z_n^j}}.
\end{equation}

Hence,~(\ref{3.1}) becomes
\begin{equation}\label{3.4}
d(z)=z^me^{c_0}e^{\delta(r)z}f_r(z).
\end{equation}
From~(\ref{3.4}), we have
\begin{equation}\label{3.5}
h_d(\theta)=\limsup_{r\rightarrow\infty}\frac{\ln|d(re^{i\theta})|}{r}=\limsup_{r\rightarrow\infty}\frac{|\delta(r)z|}{r}
+\limsup_{r\rightarrow\infty}\frac{\ln|f_r(re^{i\theta})|}{r}.
\end{equation}
For the first limit in~(\ref{3.5}). We see from~(\ref{2.18}) that $\delta(r)\sim 0.$ That is $|\delta(r)|<\delta'$ for some large $r$
for any given $\delta'>0$. Hence,
\begin{equation}\label{377}
|e^{\delta(r)z}|<e^{|\delta(r)z|}=e^{|\delta(r)| r}<e^{\delta'
r},\,\forall \delta'>0.
\end{equation}

\par
Most important of all, we approximate the infinite product
$f_r(z)$ by Levin \cite[p.112. Lemma 5;\,p.92. Theorem 2]{Levin}:
out of some zero density set we have the following asymptotic
behavior:
\begin{equation}\label{37}
\ln|{f_r(re^{i\theta})}|\sim H_1(\theta)r^{\rho(r)},\,\mbox{ where
}\rho(r)\rightarrow1,
\end{equation}
in which
\begin{equation}\label{38}
H_1(\theta)=-\int_{\theta-2\pi}^\theta[(\theta-\psi-\pi)\sin(\theta-\psi)+\cos(\theta-\psi)]d\Delta(\psi).
\end{equation}
That integral is approximated by the following sum:
\begin{equation}
S_n(\theta)=-\sum_{l=0}^n[(\theta-\psi_l-\pi)\sin(\theta-\psi_l)+\cos(\theta-\psi_l)][\Delta(\psi_{l+1})-\Delta(\psi_l)],
\end{equation}
where $\psi_0<\psi_1<\cdots<\psi_n<\psi_{n+1}$,
$|\psi_{l+1}-\psi_l|<\delta$, $l=0,1,2,\cdots,n$.
$\psi_{n+1}=\psi_0+2\pi$. We can choose $\delta$ small such that
\begin{equation}
|H_1(\theta)-S_n(\theta)|<\epsilon/3.
\end{equation}

Let $h^j(\theta)$, $H^j_1(\theta)$, $\{\psi^j_k\}$, $f^j_r(z)$ and
$S^j_n(\theta)$ be the corresponding quantities with respect to
index $n^j$. Let $0\in[\psi^1_{l_0},\psi^1_{l_0+1}]$,
$0\in[\psi^2_{l_0},\psi^2_{l_0+1}]$,
$\pi\in[\psi^1_{l_\pi},\psi^1_{l_\pi+1}]$,
$\pi\in[\psi^2_{l_\pi},\psi^2_{l_\pi+1}]$ be the angular intervals
containing nonzero density angles which are $0$ and $\pi$
by~(\ref{2.14}) and~(\ref{2.15}). Using~(\ref{2.10})
and~(\ref{2.11}),
\begin{eqnarray}
S^1_n(\theta)-S_n^2(\theta)
&=&-[(\theta-\psi^1_{l_0}-\pi)\sin(\theta-\psi^1_{l_0})+\cos(\theta-\psi^1_{l_0})][\Delta(\psi^1_{l_0+1})-\Delta^1(\psi_{l_0})]\\
&&-[(\theta-\psi^1_{l_\pi}-\pi)\sin(\theta-\psi^1_{l_\pi})+\cos(\theta-\psi^1_{l_\pi})][\Delta(\psi^1_{l_\pi+1})-\Delta^1(\psi_{l_\pi})]\\
&&-\mbox{ similar terms from }S^2(\theta).
\end{eqnarray}
Using~(\ref{2.14}) and~(\ref{2.15}), we conclude that, $\forall
n$,
\begin{equation}
S^1_n(\theta)\equiv S_n^2(\theta).
\end{equation}
Accordingly,
\begin{equation}\label{3.15}
 H^1_1(\theta)\equiv H^2_1(\theta).
\end{equation}

Similarly,~(\ref{38}) gives
\begin{equation}\label{316}
H^j_1(\theta)=\pi\Delta^j|\sin{\theta}|,
\end{equation}
with its maximal value happening at $\theta=\pm\frac{\pi}{2}$.
Also,
\begin{equation}\label{3117}
\ln|{f_r^1(re^{i\theta})}|\sim \ln|{f_r^2(re^{i\theta})}|,\mbox{
as }r\rightarrow\infty.
\end{equation}
We conclude from~(\ref{3.5}),~(\ref{377},~(\ref{37}),~(\ref{3.15})
and~(\ref{3117}),
\begin{equation}
h^1(\theta)\equiv h^2(\theta).
\end{equation}
Hence, we proved the following theorem.
\begin{proposition}\label{31}
If zeros of $d^j(z)$, $j=1,2$ inside the wedge $\Sigma_\epsilon$,
$\forall\epsilon>0$ shares the same density, then the indicator
function $h^1(\theta)\equiv h^2(\theta)$ in $[0,\pi]$. In
particular, $d^1$ and $d^2$ are entire functions of exponential
type of the same type $|a-\int_0^a \sqrt{n^1(r)}dr|=|a-\int_0^a
\sqrt{n^2(r)}dr|$.
\end{proposition}
\begin{proof}
Since they share the same indicator function, the related maximal
value of the indicator functions is the same. Hence, $d^1$, $d^2$
are entire functions of the same type by Lemma \ref{1.8}. The type
is just $|a-\int_0^a \sqrt{n^1(r)}dr|=|a-\int_0^a
\sqrt{n^2(r)}dr|$. $\Box$
\end{proof}
\begin{definition} Let $b^j:=\int_0^a \sqrt{n^j(r)}dr.$
$s^j:=a-b^j,\,j=1,2.$
\end{definition}
\begin{corollary}\label{33}
Let $h_{d^1-d^2}(\theta)$ be the indicator function related to the
$d^1(z)-d^2(z)$. Then,
\begin{equation}h_{d^1-d^2}(\pm\frac{\pi}{2})=0.\end{equation}
\end{corollary}
\begin{proof} By definition and
that  $n^1(0)=n^2(0)$, $s^1=s^2$, ~(\ref{H1}) and~(\ref{1.15})
implies
\begin{eqnarray}
h_{d^1-d^2}(\theta)=\limsup_{r\rightarrow\infty}\frac{\ln|d^1(z)-d^2(z)|}{r}=0.
\end{eqnarray}
$\Box$
\end{proof}

On the other hand, we consider the substraction in the form
of~(\ref{3.4}). Let $f_r^j(z)$, $j=1,2$, be the quantities
corresponding to $n^j(r)$ in~(\ref{3.3}). We set
\begin{eqnarray}
f_r^j(z)&:=&\prod_{\{|z_n^j|\leq
r\}}(1-\frac{z}{z_n^j})\prod_{\{|z_n^j|>r\}}(1-\frac{z}{z_n^j})e^{\frac{z}{z_n^j}}\\&:=&(\prod_{\{|z_n^j|\leq
r;\,z_n^j\in\Sigma_{\epsilon}\}}\prod_{\{|z_n^j|\leq
r;\,z_n^j\notin\Sigma_{\epsilon}\}})(1-\frac{z}{z_n^j})(\prod_{\{|z_n^j|>r;\,z_n^j\in\Sigma_{\epsilon}\}}\prod_{\{|z_n^j|>
r;\,z_n^j\notin\Sigma_{\epsilon}\}})(1-\frac{z}{z_n^j})e^{\frac{z}{z_n^j}}.
\end{eqnarray}
Using the assumption that zeros inside $\Sigma_\epsilon$ coincides
as a set, we have
\begin{eqnarray} \nonumber
d^1(z) -d^2(z)&=&z^{m^1}\prod_{|z_n^1|\leq
r;\,z_n^1\in\Sigma_{\epsilon}}(1-\frac{z}{z_n^1})\prod_{|z_n^1|>r;\,z_n^1\in\Sigma_{\epsilon}}
(1-\frac{z}{z_n^1})e^{\frac{z}{z_n^1}}\\&&\times[e^{c_0^1}e^{\delta^1(r)z}\prod_{|z_n^1|\leq
r;\,z_n^1\notin\Sigma_{\epsilon}}(1-\frac{z}{z_n^1})\prod_{|z_n^1|>
\nonumber
r;\,z_n^1\notin\Sigma_{\epsilon}}(1-\frac{z}{z_n^1})e^{\frac{z}{z_n^1}}\\&&-e^{c_0^2}e^{\delta^2(r)z}\prod_{|z_n^2|\leq
r;\,z_n^2\notin\Sigma_{\epsilon}}(1-\frac{z}{z_n^2})\prod_{|z_n^2|>
r;\,z_n^2\notin\Sigma_{\epsilon}}(1-\frac{z}{z_n^2})e^{\frac{z}{z_n^2}}].
\end{eqnarray}
Let us define
\begin{eqnarray} \label{324}\nonumber
Q_\epsilon(z)&:=&z^{m^1}[e^{c_0^1}e^{\delta^1(r)z}\prod_{|z_n^1|\leq
r;\,z_n^1\notin\Sigma_{\epsilon}}(1-\frac{z}{z_n^1})\prod_{|z_n^1|>
r;\,z_n^1\notin\Sigma_{\epsilon}}(1-\frac{z}{z_n^1})e^{\frac{z}{z_n^1}}\\&&-e^{c_0^2}e^{\delta^2(r)z}\prod_{|z_n^2|\leq
r;\,z_n^2\notin\Sigma_{\epsilon}}(1-\frac{z}{z_n^2})\prod_{|z_n^2|>
r;\,z_n^2\notin\Sigma_{\epsilon}}(1-\frac{z}{z_n^2})e^{\frac{z}{z_n^2}}],
\end{eqnarray}
which has no zero along $0i+\mathbb{R}$. Now we have
\begin{equation}\label{3.26}
d^1(z)-d^2(z)=[\prod_{|z_n^1|\leq
r;\,z_n^1\in\Sigma_{\epsilon}}(1-\frac{z}{z_n^1})\prod_{|z_n^1|>r;\,z_n^1\in\Sigma_{\epsilon}}
(1-\frac{z}{z_n^1})e^{\frac{z}{z_n^1}}]\,Q_\epsilon(z).
\end{equation}
Using~(\ref{2.18}),
\begin{equation}\label{326}
|e^{\delta^j(r)z}|\lesssim
Ce^{\delta|z|},\,\forall\delta>0,\,j=1,2.
\end{equation}
Given $\{z_n^j\}$ of zero density outside $\Sigma_\epsilon$, we
compute the following quantity.
\begin{equation}
\overline{f}_r^{j}(z):=\prod_{|z_n^j|\leq
r;\,z_n^j\notin\Sigma_{\epsilon}}(1-\frac{z}{z_n^j})\prod_{|z_n^j|>
r;\,z_n^j\notin\Sigma_{\epsilon}}(1-\frac{z}{z_n^j})e^{\frac{z}{z_n^j}}.
\end{equation}
We use formula~(\ref{37}) and~(\ref{38}).
\begin{equation}\label{328}
\ln|{\overline{f}_r^j(re^{i\theta})}|\sim
\overline{H}_1^j(\theta)r,\,\mbox{ where }r\rightarrow\infty,
\end{equation}
in which
\begin{equation}
\overline{H}_1^j(\theta)=-\int_{\theta-2\pi}^\theta[(\theta-\psi-\pi)\sin(\theta-\psi)+\cos(\theta-\psi)]d\Delta^j(\psi),
\end{equation}
outside certain exceptional set. We refer to \cite[p.112 Lemma
5]{Levin} for a complete introduction. For the application here,
we obtain that $\Delta^j(\psi)\equiv0$ for zeros outside
$\Sigma_\epsilon$. That is
\begin{equation}\label{3300}
\overline{H}_1^j(\theta)\equiv0,\,j=1,2.
\end{equation}
We note that
$\overline{H}_1^j(\theta)=h_{\overline{f}_r^{j}}(\theta)$, the
indicator function of $\overline{f}_r^j$, $j=1,2$. We refer this
to \cite[p.498]{Levin}.

\par
Using~(\ref{1.21}) and~(\ref{1.22}), we have
\begin{equation}
h_{Q_\epsilon}(\theta)\leq\max\{h_{e^{\delta^1(r)z}\overline{f}_r^1(z)},\,h_{e^{\delta^2(r)z}\overline{f}_r^2(z)}\}.
\end{equation}
Therefore,~(\ref{326}),~(\ref{328}) and~(\ref{3300}) combine to
yield $h_{Q_\epsilon}(\theta)=0$ and
\begin{equation}\label{329}
|Q_\epsilon(z)|\lesssim Ce^{\delta|z|},\,\mbox{ for some
}C>0,\,\forall\delta>0.
\end{equation}
\par
Again, the infinite product $[\prod_{|z_n^1|\leq
r;\,z_n^1\in\Sigma_{\epsilon}}(1-\frac{z}{z_n^1})\prod_{|z_n^1|>r;\,z_n^1\in\Sigma_{\epsilon}}
(1-\frac{z}{z_n^1})e^{\frac{z}{z_n^1}}]$ in~(\ref{3.26}) can be
computed similarly as the $f_r(z)$ in the~(\ref{3.3}),~(\ref{3.4})
and~(\ref{3.5}).
From~(\ref{3.5}),~(\ref{37}),~(\ref{38}),~(\ref{316}) and
Lindel\"{o}f's theorem \cite[p.28]{Levin}, we have
\begin{equation}\label{330}
h_{d^1-d^2}(\theta)=\pi\Delta^1(-\delta,\delta)|\sin\theta|.
\end{equation}
However, Corollary \ref{33} says that
$\Delta^1(-\delta,\delta)=0$. Hence, $d^1(z)-d^2(z)$ is an
exponential function of minimal type.

Therefore,~(\ref{329}) and~(\ref{330}) give
\begin{equation}\label{3.29}
|d^1(z)-d^2(z)|\lesssim Ce^{\delta|z|},\,\forall\delta>0.
\end{equation}
In addition to that, using~(\ref{1.12}),
\begin{equation}\label{3.30}
d^1(x)-d^2(x)\rightarrow0,\mbox{ as }x\rightarrow0i\pm\infty.
\end{equation}
Using Phragm\'{e}n-Lindel\"{o}f theorem as in Titchmarsh
\cite[p.178]{Titchmarsh},~(\ref{3.29}) and~(\ref{3.30}) imply that
$d^1(z)-d^2(z)$ is bounded both in lower and half complex planes.
Any bounded entire function is a constant which is zero as
suggested by~(\ref{3.30}). Hence,
\begin{equation}d^1(z)\equiv d^2(z).
\end{equation}
\par
Finally, following the argument of Aktosun, Gintides and
Papanicolaou around (3.7), (3.8), (3.9), (3.10) in \cite[Corollary
2.10]{Aktosun}, we see that
\begin{equation}
d(z)=\frac{1}{a^2}\{\frac{\sin{za}}{z}y'(a)-\cos(za)y(a)\}.
\end{equation}
\par
Consider a substraction for two different $n^j(r)$'s, we obtain
\begin{equation}
d^1(z)-d^2(z)=\frac{1}{a^2}\{\frac{\sin{za}}{z}[{y^1}'(a)-{y^2}'(a)]-\cos(za)[y^1(a)-y^2(a)]\}\equiv0.
\end{equation}
Let $za=n\pi$, $n\in\mathbb{Z}$. In this case,
\begin{equation}\label{335}
y^1(a;\frac{n\pi}{a})=y^2(a;\frac{n\pi}{a}),\,n\in\mathbb{Z}.
\end{equation}
Similarly, let $za=\frac{n\pi}{2}$, $n$, odd. We obtain
\begin{equation}
{y^1}'(a;\frac{n\pi}{2a})={y^2}'(a;\frac{n\pi}{2a}),\,n\mbox{ odd
}.
\end{equation}
Again, $y^j(a;z)$'s are entire functions of exponential type. We
use a generalized Carlson's theorem from Levin
\cite[p.190]{Levin}. This is a substitute for
Phragm\'{e}n-Lindel\"{o}f theorem.
\begin{theorem}\label{3.2}
Let $F(z)$ be holomorphic and at most of normal type with respect
to the proximate order $\rho(r)$ in the angle $\alpha\leq\arg
z\leq\alpha+\pi/\rho$ and vanish on a set $N:=\{a_k\}$ in this
angle, with angular density $\Delta_N(\psi)$. Let
$$
H_N(\theta):=\pi\int_{\alpha}^{\alpha+\pi/\rho}\sin|\psi-\theta|d\Delta_N(\psi),
$$
when $\rho$ is integral. Then, if $F(z)$ is not identically zero,
\begin{equation}
h_F(\alpha)+h_F(\alpha+\pi/\rho)\geq
H_N(\alpha)+H_N(\alpha+\pi/\rho).
\end{equation}
\end{theorem}
Now let \begin{equation} F(z):=y^1(a;z)-y^2(a;z),
\end{equation}
\begin{equation}
G(z):={y^1}'(a;z)-{y^2}'(a;z)
\end{equation} and
$\rho\equiv1,\,\alpha=-\frac{\pi}{2}.$

\par
We deal with $F(z)$ firstly, it has a set of common zeros of
density $\Delta^N$ and supported only at $\psi=0,\pi$. Hence,
\begin{equation}\label{3.20}
H_N(\theta)=\pi\int_{-\frac{\pi}{2}}^{\frac{\pi}{2}}\sin|\psi-\theta|d\Delta_N(\psi)=\Delta^N\pi|\sin\theta|,
\,\theta\in[-\frac{\pi}{2},\frac{\pi}{2}],
\end{equation}
where $N$ is the set of common zeros as described by~(\ref{335}).
Accordingly, we may compute from~(\ref{335}) that
$\Delta^N=\frac{a}{\pi}$. We refer to \cite[p.91;\,Ch 2.
Sec\,3.]{Levin} for a systematic and a step by step calculation.
Besides that, we need to compute the indicator function
$h_{F}(\theta)$. Recalling from~(\ref{1.9}),
\begin{equation}
y^j(a;z)=\frac{1}{[\epsilon_1^j(0)]^{\frac{1}{4}}z}\sin[z\int_0^a\sqrt{\epsilon_1^j(\rho)}d\rho][1+O(\frac{1}{z})],\,j=1,2.
\end{equation}
Therefore,
\begin{equation}\nonumber
\ln|y^1(a;z)-y^2(a;z)|=\ln|\frac{1}{z[\epsilon_1^1(0)]^{\frac{1}{4}}}|+\ln|\sin[z\int_0^a\sqrt{\epsilon_1^1(\rho)}d\rho]
-\sin[z\int_0^a\sqrt{\epsilon_1^2(\rho)}d\rho]|+\ln|1+O(\frac{1}{z})|.
\end{equation}
Using~(\ref{1.9}),~(\ref{H1}), Proposition \ref{31} and Definition
\ref{1.5}, we obtain
\begin{equation}
h_{F}(\theta)=0,\,\theta\neq 0,\,\pi.
\end{equation}
Again, $h_F(\theta)$ is a continuous function given $F$ is entire.
Hence, we have
\begin{equation}\label{3.23}
h_{F}(\theta)=0,\,\theta\in[0,2\pi].
\end{equation}
Combining Theorem \ref{3.2},~(\ref{3.20}) and~(\ref{3.23}), we
obtain $F(z)\equiv 0$.

\par
Similarly, we can prove $G(z)\equiv0$. In particular, we have
shown
\begin{equation}
y^1(a;z)\equiv y^2(a;z);\,{y^1}'(a;z)\equiv{y^2}'(a;z).
\end{equation}
The last ingredient is applying the uniqueness for the following
inverse Sturm-Liouville problem
\begin{eqnarray}
&&\psi''(x)+k^2\rho(x)\psi(x)=0,\,0<x<a;\\
&&\psi(0)=\psi(a)=0.
\end{eqnarray}
as did in \cite[Corollary 2.10]{Aktosun}. We conclude that
\begin{equation}\label{3.27}
\epsilon^1_1(r)\equiv \epsilon^2_1(r).
\end{equation}

\end{document}